\documentclass[11 pt]{article}

\usepackage{amsfonts,amscd,amssymb,amsthm}
\usepackage{amsmath}
\usepackage{amsxtra}
\usepackage{amscd}
\usepackage{eucal}
\usepackage{array}
\usepackage{graphicx}
\usepackage{tikz}
\usepackage{url}

\newtheorem{theorem}{Theorem}[section]
\newtheorem{lemma}[theorem]{Lemma}

\theoremstyle{definition}
\newtheorem{remark}[theorem]{Remark}
\newtheorem{example}[theorem]{Example}

\title{Multigraded Betti numbers of simplicial forests }
\author{Nursel Erey and Sara Faridi}

\begin{document}
\date{}
\maketitle


\begin{abstract}
We prove that multigraded Betti numbers of a simplicial forest are
always either $0$ or $1$. Moreover a nonzero multidegree appears
exactly in one homological degree in the resolution. Our work
generalizes work of Bouchat~\cite{bouchat} on edge ideals of graph trees.

\end{abstract}

\bigskip

\section{Introduction}
The Betti numbers of edge ideals of graph forests were studied by
several authors (\cite{bouchat}, \cite{ha van
  tuyl}, \cite{Jacques}, \cite{jacques katzman
  forests}, \cite{kimura}). Kimura~\cite{kimura}
combinatorially characterized the graded Betti numbers for a graph
forest. In~\cite{bouchat} Bouchat proved that multigraded Betti
numbers of graph trees are always $0$ or $1$ by using the mapping cone
construction. Ehrenborg and Hetyei~\cite{ehrenborg} showed that the
independence complex of graph forests are simple-homotopy equivalent
to a single vertex or to a sphere. By the well-known formula of
Hochster this implies that multigraded Betti numbers of graph forests
appear in at most one homological degree. We shall generalize these
results about multigraded Betti numbers to simplicial forests.

Note that multigraded Betti numbers of edge ideals are not necessarily
$0$ or $1$ in general. Also a multidegree can appear in more than one
homological degree, see the example below.

\begin{example} For $I= ( ab, ae, be, cd, ce, de)$ one can check with 
Macaulay$2$~\cite{m2} that $b_{1,abe}(I)=b_{1,cde}(I)=2$ and $b_{2,abcde}(I)=b_{3,abcde}(I)=1$. 
\end{example}

\section{Background material}
\subsection{Resolutions}
Let $S=\Bbbk[x_1,\ldots,x_n]$ be the polynomial ring in $n$ variables
over a field $\Bbbk$. A \textbf{minimal free resolution} of a monomial
ideal $I$ is an exact sequence of free $S$-modules
$$ 0 \longrightarrow F_r \overset{d_r}{\longrightarrow} \cdots\longrightarrow F_1 \overset{d_1}{\longrightarrow} F_0  \overset{d_0}{\longrightarrow} I \longrightarrow 0   $$
such that $d_{i+1}(F_{i+1})\subseteq (x_1,\ldots,x_n)F_i$ for all $i\geq 0$. The rank of $F_i$ is called the $i$th \textbf{total Betti number} of $I$ and is denoted by $b_{i}^S(I)$. Moreover if the differential maps preserve the (standard) degrees, then the resolution is called a \textbf{minimal graded free resolution}. In this case the resolution is of the form

$$ 0 \longrightarrow \bigoplus_{j \in \mathbb{N}} S(-j)^{b_{r,j}^S(I)} \overset{d_r}{\longrightarrow} \cdots\longrightarrow \bigoplus_{j \in \mathbb{N}} S(-j)^{b_{1,j}^S(I)} \overset{d_1}{\longrightarrow} \bigoplus_{j \in \mathbb{N}} S(-j)^{b_{0,j}^S(I)} \overset{d_0}{\longrightarrow} I \longrightarrow 0   $$ where the integers $b_{i,j}^S(I)$ are the \textbf{graded Betti numbers} of $I$.

One usually also considers $\mathbb{N}^n$-grading (multigrading) on $S$ where $\mathbb{N}=\{0, 1, 2,\ldots\}$. Note that with this grading the degree of a monomial $m=x_1^{a_1}x_2^{a_2}\cdots x_n^{a_n}$ is equal to $\bold{m}=(a_1,\ldots,a_n).$ If the differential maps of a minimal free resolution preserve the multidegrees, then it takes the following form:
$$ 0 \longrightarrow \bigoplus_{\bold{m} \in \mathbb{N}^n} S(-\bold{m})^{b_{r,\bold{m}}^S(I)} \overset{d_r}{\longrightarrow} \cdots\longrightarrow \bigoplus_{\bold{m} \in \mathbb{N}^n} S(-\bold{m})^{b_{1,\bold{m}}^S(I)} \overset{d_1}{\longrightarrow} \bigoplus_{\bold{m} \in \mathbb{N}^n} S(-\bold{m})^{b_{0,\bold{m}}^S(I)} \overset{d_0}{\longrightarrow} I \longrightarrow 0   $$
 which is called a \textbf{minimal multigraded free resolution}. The associated ranks $b_{i,\bold{m}}^S(I)$ are called \textbf{multigraded Betti numbers} of $I$.

Clearly the Betti numbers are related with the following equations.
$$ b_{i}^S(I)=\sum_{j \in \mathbb{N}}b_{i,j}^S(I) $$
\begin{equation}\label{definition of graded betti number} 
b_{i,j}^S(I)=\sum_{\deg(m)=j}b_{i,\bold{m}}^S(I) 
\end{equation} 
where $\deg(m)$ stands for the standard degree of $m$, i.e.,
$\deg(x_1^{a_1}\cdots x_n^{a_n})=a_1+\cdots+a_n.$ For simplicity, we shall
use a monomial $m$ and its $\mathbb{N}^n$-degree $\bold{m}$
interchangeably.

\subsection{Simplicial complexes and homology}
An \textbf{abstract simplicial complex} $\Gamma$ on a set of
\textbf{vertices} $\mathcal{V}(\Gamma)=\{x_1,\ldots,x_n\}$ is a
collection of subsets of $\mathcal{V}(\Gamma)$ such that $\{x_i\}\in
\Gamma$ for all $i$ and, $F\in \Gamma$ implies that all subsets of $F$
are also in $\Gamma .$ The elements of $\Gamma$ are called
\textbf{faces} and the maximal faces under inclusion are called
\textbf{facets}.

Since the simplicial complex $\Gamma$ is determined by its facets
$F_1,\ldots,F_q$ we say that  $F_1,\ldots,F_q$ \textbf{generate} $\Gamma$
and, write $ \Gamma= \langle F_1,\ldots,F_q\rangle $ or
$\operatorname{Facets}(\Gamma)=\{F_1,\ldots,F_q\}$.  A
\textbf{subcollection} of $\Gamma$ is a subcomplex generated by a
subset of the facets of $\Gamma$. The simplicial complex obtained by
\textbf{removing the facet} $F_i$ from $\Gamma$ is the simplicial
complex $\Gamma \setminus \langle F_i \rangle = \langle
F_1,\ldots,\widehat{F_i},\ldots,F_q \rangle $. If $A$ is a subset of
$\mathcal{V}(\Gamma)$, \textbf{the induced subcollection on $A$} is
defined as $\Gamma_{A}=\langle F\in \operatorname{Facets}(\Gamma) \mid
F\subseteq A \rangle$.

Two facets $F$ and $G$ of $\Gamma$ are \textbf{connected} if there
exists a chain of facets of $\Gamma$, $F_0=F,F_1,\ldots,F_m=G$ such that
every two consecutive facets have nonempty intersection. The simplicial
complex $\Gamma$ is called \textbf{connected} if any two of its facets
are connected.

Let $F$ be a facet of $\Gamma$. The \textbf{connected component} of
$F$ in $\Gamma$ is denoted by $\operatorname{conn}_{\Gamma}(F)$. If
$\operatorname{conn}_{\Gamma}(F)\setminus \langle F \rangle= \langle F_1,\ldots,F_p
\rangle $, then the \textbf{the reduced connected component of} $F$ in
$\Gamma$ denoted by $\operatorname{\overline{conn}}_{\Gamma}(F)$ will
be the simplicial complex
$$\operatorname{\overline{conn}}_{\Gamma}(F)= \langle F_i\setminus F
 \mid (F_j \setminus F) \nsubseteq (F_i\setminus F), \ 
j\neq i, \ 1\leq i, j \leq p \rangle. $$ In other words, the facets of
$\operatorname{\overline{conn}}_{\Gamma}(F)$ are the minimal nonempty
sets among all sets $G \setminus F$, where $G$ is a facet of $\operatorname{conn}_{\Gamma}(F)$.

A facet $F$ of $\Gamma$ is a \textbf{leaf} if either $F$ is the only
facet of $\Gamma$, or there exists a facet $G\in \Gamma$ such that
$G \neq F$ and $F\cap F' \subseteq G$ for every facet $F'\neq F$. By definition,
every leaf $F$ of $\Gamma$ contains a vertex $v$ such that $v \notin
F'$ for every facet $F'\neq F$ of $\Gamma$. Such a vertex is called a
\textbf{free vertex}. A connected simplicial complex $\Gamma$ is a
\textbf{tree} if every nonempty subcollection of $\Gamma$ has a
leaf. We say $\Gamma$ is a \textbf{forest} if every connected
component of $\Gamma$ is a tree.

The \textbf{facet ideal} $\mathcal{F}(\Gamma)$ of $\Gamma$ is the monomial ideal in $\Bbbk[x_1,\ldots,x_n]$ which is generated by $\{\prod_{x\in F}x \mid F \text{ is a facet of } \Gamma \}.$ Using the following correspondence 
$$m= x_{i_1}\cdots x_{i_s}, \text{ a squarefree monomial} \Leftrightarrow A=\{x_{i_1},\ldots,x_{i_s}\}\subseteq \mathcal{V}(\Gamma)$$ we shall use the squarefree monomials and nonempty subsets of $\mathcal{V}(\Gamma)$ interchangeably.

A \textbf{simplex} is a simplicial complex with only one nonempty facet. For each integer $ i $, the $ \Bbbk $ -vector space $ \widetilde{H}_i(\Gamma,\Bbbk) $ is the $i$th \textbf{reduced homology} of $ \Gamma $ over $ \Bbbk $.


\subsection{The Taylor complex}
Let $I$ be an ideal in $S=\Bbbk[x_1,\ldots,x_n]$ which is minimally
generated by the monomials $m_1,\ldots,m_s$. In~\cite{taylor}, Taylor
constructed an explicit multigraded free resolution of $I$ which is
usually nonminimal.  This construction was generalized then to
simplicial resolutions in~\cite{bayer peeva sturmfels}. Taylor's
resolution is an example of a simplicial resolution where the
underlying simplicial complex is a full simplex over the vertex set
labeled with $\{m_1,\ldots,m_s\}$, called the \textbf{Taylor simplex} of
$I$. The Betti numbers of $I$ can be determined by the dimensions of
reduced homologies of certain subcomplexes of the Taylor
simplex. Before stating this precisely we need one more definition.

Let $\Theta$ be the Taylor simplex whose vertices are labeled with
monomials $m_1,\ldots,m_s$. If $\tau=\{m_{i_1},\ldots,m_{i_r}\}$ is a face
of $\Theta$, then by $\operatorname{lcm}(\tau)$ we mean
$\operatorname{lcm}(m_{i_1},\ldots,m_{i_r})$.  For any monomial $m$ in
$S$ the simplicial subcomplex $\Theta_{<m}$ is defined as
$$ \Theta_{<m}= \{\tau \in \Theta \mid \operatorname{lcm} (\tau)
\text{ strictly divides } m\}.  $$

\begin{example}\label{example of Taylor simplex} 
For $I=(x_1x_2, x_1x_3, x_1x_4, x_3x_4) $ the Taylor
simplex $\Theta$ and a subcomplex $\Theta_{<x_1x_2x_3x_4}$ are
illustrated in Figures~\ref{fig:minipage1} and~\ref{fig:minipage2}
respectively.
\end{example}

\begin{figure}[ht]
\begin{minipage}[b]{0.45\linewidth}
\quad
\quad
\quad
\begin{tikzpicture}
[scale=1, vertices/.style={draw, fill=black, circle, inner sep=0.5pt}]
\filldraw[fill=black!10, draw=black] (2,0.5) -- (3,0) -- (2.9,1.5) -- cycle; 
\filldraw[fill=black!10, draw=black] (3,0) -- (2.9,1.5) -- (3.8,0.8) -- cycle;
\node[vertices, label=below:{$x_3x_4$}] (a) at (3,0) {};
\node[vertices, label=right:{$x_1x_4$}] (b) at (3.8,0.8) {};
\node[vertices, label=left:{$x_1x_3$}] (c) at (2,0.5) {};
\node[vertices, label=above:{$x_1x_2$}] (d) at (2.9,1.5) {};
\foreach \to/\from in {a/b,a/c,a/d,b/d,c/d}
\draw [-] (\to)--(\from);
\draw [dashed] (b)--(c);
\end{tikzpicture}
\caption{A Taylor simplex $\Theta$}
\label{fig:minipage1}
\end{minipage}
\quad 
\begin{minipage}[b]{0.45\linewidth}
\quad
\quad
\quad
\begin{tikzpicture}
[scale=1, vertices/.style={draw, fill=black, circle, inner sep=0.5pt}]
\filldraw[fill=black!10, draw=black] (-1,0.5) -- (0,0) -- (0.8,0.8) -- cycle; 

\node[vertices, label=below:{$x_3x_4$}] (a) at (0,0) {};
\node[vertices, label=right:{$x_1x_4$}] (b) at (0.8,0.8) {};
\node[vertices, label=left:{$x_1x_3$}] (c) at (-1,0.5) {};
\node[vertices, label=above:{$x_1x_2$}] (d) at (-0.1,1.5) {};
\foreach \to/\from in {a/b,a/c,b/d,c/d}
\draw [-] (\to)--(\from);
\draw [dashed] (b)--(c);
\end{tikzpicture}
\caption{$\Theta_{<x_1x_2x_3x_4}$}
\label{fig:minipage2}
\end{minipage}
\end{figure}

\begin{theorem}[\cite{bayer peeva sturmfels}]\label{betti number by homology} 
Let I be a proper monomial ideal of $S$ which is minimally generated by the
monomials $m_1,\ldots,m_s$.  Denote by $\Theta$ the Taylor simplex of
$I$. For $i \geq 0$, the multigraded Betti numbers of $I$ are given by
 \begin{equation}
     b_{i,m}^S (I)=
    \begin{cases}
      \dim_{\Bbbk} \widetilde{H}_{i-1}(\Theta_{<m};\Bbbk), &
      \text{if}\ m \text{ divides
      }\ \operatorname{lcm}(m_1,\ldots,m_s)\\ 0, & \text{otherwise.}
    \end{cases}
  \end{equation}
\end{theorem}

\begin{remark}By Theorem~\ref{betti number by homology}, we are allowed not to 
specify a polynomial ring $S$ when we deal with Betti numbers. We can
think of a facet ideal $\mathcal{F}(\Gamma)$ lying in a polynomial
ring over $\Bbbk$ that contains at least as many variables as the
vertices of $\Gamma$. Therefore we drop $S$ and write
$b_{i,m}(\mathcal{F}(\Gamma))$ and $b_{i,j}(\mathcal{F}(\Gamma))$ for
the Betti numbers.
\end{remark}

\begin{remark}\label{r:top-betti} If $I=(m_1,\ldots,m_s)$ and 
$q=\deg\operatorname{lcm}(m_1,\ldots,m_s)$ then for any $r > q$ we have
  $b_{i,r}(I)=0$ for all $i$. That is, $q$ is the largest possible
  grade at which the Betti number can be nonzero. Therefore we call
  the numbers $b_{i,q}(I), i \in \mathbb{Z}$ as the \textbf{top grade
    Betti numbers}. Clearly, for a facet ideal $\mathcal{F}(\Gamma)$,
  the top grade is the number of vertices of $\Gamma$.
\end{remark}

\begin{remark}\label{monomial ideal generated by single element}If $m$ is 
  one of the minimal generators of I, then $b_ {i,m} (I) = 1$ when $i
  = 0$ and is zero otherwise. If $I$ is generated by a single monomial
  $m$, then its multigraded resolution is
$$ 0 \rightarrow S(-m) \rightarrow I \rightarrow 0.$$
\end{remark}


\section{Betti numbers of simplicial forests}

\begin{lemma}\label{multigraded betti number and induced subgraph} If $m$ is a squarefree monomial of degree $j$, then $b_{i,m}(\mathcal{F}(\Gamma))=b_{i,j}(\mathcal{F}(\Gamma_m))$.
\end{lemma}
\begin{proof}Let $\Theta$ and $\Lambda$ be Taylor simplices of 
$\mathcal{F}(\Gamma)$ and $\mathcal{F}(\Gamma_m)$ respectively. Then
  clearly we have $\Theta_{<m} = \Lambda_{<m}$. So by
  Theorem~\ref{betti number by homology}, $
  b_{i,m}(\mathcal{F}(\Gamma))=b_{i,m}(\mathcal{F}(\Gamma_m))$. But by
  Equation (\ref{definition of graded betti number}) we get
  $b_{i,m}(\mathcal{F}(\Gamma_m)) = b_{i,j}(\mathcal{F}(\Gamma_m))$
  since $m$ is the only possible squarefree monomial of degree $j$
  that can divide the $\operatorname{lcm}$ of the generators of
  $\mathcal{F}(\Gamma_m)$.
\end{proof}

\begin{lemma}\label{lemma betti numbers of connected components}If $I_1, I_2,\ldots,I_N$ are squarefree monomial ideals whose minimal generators contain no common variable, then for $i,j \geq 0$

\begin{equation} \label{betti numbers of connected components}
b_{i,j}\left(S/(I_1+I_2+\cdots+I_N)\right)= \sum_{\substack{u_1+\cdots+u_N=i
    \\ v_1+\cdots+v_N=j}} b_{u_1,v_1}(S/I_1)\cdots b_{u_N,v_N}(S/I_N).
\end{equation}
Moreover, if the least common multiple of the minimal generators of each $I_r$ is of degree $q_r$, then
\begin{equation}\label{top betti numbers of connected components}
b_{i, q_1+\cdots+q_N}\left(S/(I_1+I_2+\cdots+I_N)\right)=  \sum_{u_1+\cdots+u_N=i } b_{u_1,q_1}(S/I_1)\cdots b_{u_N,q_N}(S/I_N).
\end{equation}
\end{lemma}

   \begin{proof} The case $N=2$ of 
    Equation~(\ref{betti numbers of connected components}) is
    Corollary~$2.2$ of~\cite{jacques katzman forests}, and the general
    case follows from an easy induction on $N$. To see (\ref{top betti
      numbers of connected components}), note that we have
     $$b_{i,q_1+\cdots+q_N}\left(S/(I_1+I_2+\cdots+I_N)\right)=
    \sum_{\substack{u_1+\cdots+u_N=i \\ v_1+\cdots+v_N=q_1+\cdots+q_N}}
    b_{u_1,v_1}(S/I_1)\cdots b_{u_N,v_N}(S/I_N)$$ by Equation~(\ref{betti
      numbers of connected components}). Suppose that
    $v_1+\cdots+v_N=q_1+\cdots+q_N$. If $v_{\ell} \neq q_{\ell}$ for some
    $\ell$, then there exists a $j$ such that $v_j > q_j$ whence
    $b_{u_j,v_j}(S/I_j)=0$ since $b_{u_j,q_j}$ is a top grade Betti
    number. In this case the term $
    b_{u_1,v_1}(S/I_1)\cdots b_{u_N,v_N}(S/I_N)$ vanishes. So we can rewrite
    the sum above as
$$ \sum_{\substack{u_1+\cdots+u_N=i  \\ v_1=q_1,\ldots,v_N=q_N}} b_{u_1,v_1}(S/I_1)\cdots b_{u_N,v_N}(S/I_N)$$
and this completes the proof. 
\end{proof}

We will make use of the following results on simplicial trees.

\begin{lemma}[\cite{ha van tuyl},~\cite{the facet ideal of a simplicial complex}]\label{reduced connected component is forest} Let $F$ be a facet of a forest $\Gamma$.
 Then $\overline{\operatorname{conn}}_{\Gamma}(F)$ is a forest.
\end{lemma}

If $\Gamma$ is a simplicial tree, one can order its facets as
$F_0,F_1,\ldots,F_q$ so that each facet $F_i$ is a leaf of the simplicial
tree $\Gamma_{i}=\langle F_0,\ldots,F_i \rangle$ for $0 \leq i \leq q$.
In~\cite{a good leaf order on simplicial trees}, based on such an
order, a refinement of the recursive formula for graded Betti numbers
of simplicial forests \cite[Theorem 5.8]{ha van tuyl} of H\`{a} and
Van Tuyl was given.

\begin{theorem}[Proposition 4.9,~\cite{a good leaf order on simplicial trees}]\label{refined recursive formula}
 Let $\Gamma$ be a simplicial tree whose facets $F_0,F_1,\ldots,F_q$ are
 ordered such that each facet $F_i$ is a leaf of the simplicial tree
 $\Gamma_{i}=\langle F_0,\ldots,F_i \rangle$ for $0 \leq i \leq q$. Then
 for all $i\geq 1$ and $j \geq 0$
$$b_{i,j}(\mathcal{F}(\Gamma))= b_{i,j}(\mathcal{F}(\langle F_0
 \rangle))+ \sum_{u=1}^q b_{i-1,j-\lvert F_u
   \rvert}\big(\mathcal{F}(\overline{\operatorname{conn}}_{\Gamma_u}(F_u))\big) $$
 where we adopt the convention that $b_{-1,j}(I)$ is $1$ if $j=0$ and
 is $0$ otherwise for any ideal $I$.
\end{theorem}


We now prove the main result of this paper.

\begin{theorem} Let $\Gamma$ be a simplicial forest. Then multigraded Betti numbers 
of $\mathcal{F}(\Gamma)$ are either $0$ or $1$. Moreover, if for some
monomial $m$ we have $b_{i,m}(\mathcal{F}(\Gamma))\neq 0$, then
$b_{h,m}(\mathcal{F}(\Gamma))= 0$ for all $h \neq i$.
\end{theorem}

     \begin{proof} We prove the given statements by induction on the number 
      of vertices of $\Gamma$. The cases when $\Gamma$ has only one
      vertex or $i = 0$ are  clear by Remark~\ref{monomial ideal generated by
        single element}. Suppose that the given statements hold for
      any simplicial forest whose number of vertices is $s$ or
      less. Now let $\Gamma$ be a simplicial forest on $s+1$ vertices
      and take a monomial $m$ which divides the product of vertices of
      $\Gamma$. Then the induced subcollection $\Gamma_m$ is also a
      forest by definition. Note that we have $b_{i,m}
      (\mathcal{F}(\Gamma)) = b_{i,j} (\mathcal{F}(\Gamma_m))$ by
      Lemma~\ref{multigraded betti number and induced subgraph} where
      $j=\operatorname{deg}(m)$. If $j$ is greater than the number of
      vertices of $\Gamma_m$, then $b_{i,m} (\mathcal{F}(\Gamma)) =
      b_{i,j} (\mathcal{F}(\Gamma_m))=0$ by Remark~\ref{r:top-betti}.
      So we assume $|\mathcal{V}(\Gamma_m)|=j=\operatorname{deg}(m)$.

      If $\Gamma_m$ is not a tree,
      then its connected components $\Upsilon_1,\ldots,\Upsilon_t$
      satisfy the induction hypothesis.

      If $\mathcal{F}(\Gamma_{m})=0$ we have nothing to prove. So we
      assume that $\mathcal{F}(\Gamma_{m})\neq 0$, and using
      Lemma~\ref{lemma betti numbers of connected components} we get
      $$b_{i,j}(\mathcal{F}(\Gamma_{m}))= b_{i+1,j}(S/\mathcal{F}(\Gamma_{m}))= \sum_{\gamma_1+\cdots+\gamma_t=i+1}
      b_{\gamma_1,l_1}(S/\mathcal{F}(\Upsilon_1))\dots b_{\gamma_t,l_t}
      (S/\mathcal{F}(\Upsilon_t))$$ where $l_v$ is the number of
      vertices of $\Upsilon_v$ for each $1\leq v \leq t$. As each connected component has at least one vertex, $\mathcal{F}(\Upsilon_v)\neq 0$ for each $v$. By induction hypothesis for each $l_v$ there exists at most one $\gamma_v$ such that $b_{\gamma_v,l_v}(\mathcal{F}(\Upsilon_v))\neq 0$. Therefore for each $l_v$ there exists at most one $\gamma_v$ such that $b_{\gamma_v,l_v}(S/\mathcal{F}(\Upsilon_v))\neq 0$.

Hence we see that there must be at most one $i$ such that
      $b_{i,j}(\mathcal{F}(\Gamma_{m}))\neq 0$. And, in such a case
      $$b_{i,j}(\mathcal{F}(\Gamma_{m}))=
      \prod_{v=1}^tb_{\gamma_v,l_v}(S/\mathcal{F}(\Upsilon_v))= \prod_{v=1}^t b_{\gamma_v-1,l_v}(\mathcal{F}(\Upsilon_v))=\prod_{v=1}^t
      1=1$$ as desired. Therefore we assume that $\Gamma_{m}$ is a tree
      and $j=|\mathcal{V}(\Gamma_{m})|$.

      Suppose that the facets $F_0,F_1,\ldots,F_q$ of $\Gamma_m$ are
      ordered as in Theorem~\ref{refined recursive formula}.  Then we
      have $j=\lvert \cup _{r=0}^{q}F_r\rvert$ as $\Gamma_m$ is a
      simplicial complex on $j$ vertices. Now we have
      \begin{equation}\label{equation in proof of multigraded simplicial forest}
        b_{i,j} (\mathcal{F}(\Gamma_m)) = b_{i,j}(\mathcal{F}(\langle
        F_0 \rangle))+ \sum_{u=1}^q b_{i-1,j-\lvert F_u
          \rvert}\big(\mathcal{F}(\overline{\operatorname{conn}}_{(\Gamma_m)_u}(F_u))\big)
      \end{equation}
      by Theorem~\ref{refined recursive formula}. If $F_0$ is the only
      facet of $\Gamma_m$, then we are done by Remark~\ref{monomial ideal
        generated by single element}. So assume that $q\geq 1$ and
      note that the set of facets of
      $\overline{\operatorname{conn}}_{(\Gamma_m)_u}(F_u)$ is a subset
      of $\{F_0\setminus F_u,\ldots, F_{u-1}\setminus F_u \}$ for every
      $1 \leq u \leq q$.

      Since $F_q$ has a free vertex in $\Gamma_m, \lvert
      \mathcal{V}((\Gamma_m)_u) \rvert < j$ for $u <q$. In particular,
      $\lvert F_0 \rvert < j$ and
      $\lvert\mathcal{V}(\overline{\operatorname{conn}}_{(\Gamma_m)_u}(F_u))\rvert
      < j-\lvert F_u \rvert$ when $u<q$. Hence by
      Remark~\ref{r:top-betti}, Equation~(\ref{equation in proof of
        multigraded simplicial forest}) turns into $$b_{i,m}
      (\mathcal{F}(\Gamma))= b_{i-1, j-\lvert F_q \rvert}
      (\mathcal{F}(\overline{\operatorname{conn}}_{(\Gamma_m)_q}(F_q))). $$
      Observe that by definition of
      $\overline{\operatorname{conn}}_{(\Gamma_m) _q}(F_q) $ some of
      $F_0 \setminus F_q,\ldots,F_{q-1} \setminus F_q$ might have already
      been omitted when forming the facet set of
      $\overline{\operatorname{conn}}_{(\Gamma_m) _q}(F_q)$. So, $ j -
      \lvert F_q \rvert $ is greater than or equal to the number of
      vertices of $\overline{\operatorname{conn}}_{(\Gamma_m)
        _q}(F_q)$. If it is greater, then
       $$b_{i,m} (\mathcal{F}(\Gamma))= b_{i-1, j-\lvert F_q
        \rvert}(\mathcal{F}(\overline{\operatorname{conn}}_{(\Gamma_m)_q}(F_q)))=0 $$
      and nothing is left to prove. Otherwise,
      $\overline{\operatorname{conn}}_{(\Gamma_m)_q}(F_q)$ is a
      simplicial forest on $j-\lvert F_q \rvert$ vertices by
      Lemma~\ref{reduced connected component is forest}. Since $j \leq
      s+1$, $\overline{\operatorname{conn}}_{(\Gamma_m)_q}(F_q)$
      satisfies the induction hypothesis. The proof follows by
      observing that $b_{i-1, j-\lvert F_q
        \rvert}(\mathcal{F}(\overline{\operatorname{conn}}_{(\Gamma_m)_q}(F_q)))$
      is also a multigraded Betti number.

\end{proof}


\begin{thebibliography}{9}
\bibitem{bayer peeva sturmfels} D. Bayer, I. Peeva, and B. Sturmfels, \emph{Monomial resolutions}, Math. Res. Lett. 5 (1998), 31--46.
\bibitem{bouchat}R. R. Bouchat, \emph{Free resolutions of some edge ideals of simple graphs}, Journal of Commutative Algebra, Volume 2, No.1, Spring 2010.
\bibitem{ehrenborg} R. Ehrenborg, G. Hetyei, \emph{The topology of the independence complex}, European Journal of Combinatorics 27 (2006) 906--923.
\bibitem{a good leaf order on simplicial trees}S. Faridi, \emph{A good
    leaf order on simplicial trees}, Connections Between Algebra,
  Combinatorics, and Geometry, in: Springer Proceedings in Mathematics
  and Statistics, vol. 76, Springer-Verlag, 2014.
\bibitem{the facet ideal of a simplicial complex}S. Faridi, \emph{The facet ideal of a simplicial complex}, Manuscripta Math. 109 (2002) 159--174.
\bibitem{m2} D. R. Grayson, M. E. Stillman, Macaulay2, \emph{a software system for research in algebraic geometry}, available at http: \url{www.math.uiuc.edu/Macaulay2/}. 
\bibitem{ha van tuyl}H. T. H\`{a}, A. Van Tuyl, \emph{Splittable ideals and the resolutions of monomial ideals}, Journal of Algebra 309 (2007) 405--425.
\bibitem{Jacques} S. Jacques, \emph{Betti numbers of graph ideals}, PhD thesis, University of Sheffield, math.AC/0410107 (2004).
\bibitem{jacques katzman forests} S. Jacques, M. Katzman, \emph{The Betti numbers of forests}, math.AC/0410107, 2004.
\bibitem{kimura}K. Kimura, \emph{Non-vanishingness of Betti numbers of edge ideals}, Harmony of Gr\"obner bases and the modern industrial society, 153--168, World Sci. Publ., Hackensack, NJ, 2012.
\bibitem{taylor} D. Taylor, \emph{Ideals generated by monomials in an $R$-sequence}, PhD Thesis, University of Chicago (1966).










\end{thebibliography}
\end{document}